\documentclass[12pt]{amsart}
\usepackage{amssymb,mathrsfs,longtable}
\usepackage[matrix,arrow,curve]{xy}
\usepackage{setspace}
\usepackage{multirow}
\usepackage{epsfig,color}
\usepackage{textcomp}
\usepackage{url}

\newcounter{cequation}[section]

\newtheorem{theorem}[cequation]{Theorem}
\newtheorem*{theorem*}{Theorem}
\newtheorem{lemma}[cequation]{Lemma}
\newtheorem{corollary}[cequation]{Corollary}
\newtheorem{conjecture}[cequation]{Conjecture}
\newtheorem{question}[cequation]{Question}
\newtheorem{proposition}[cequation]{Proposition}

\theoremstyle{definition}
\newtheorem{example}[cequation]{Example}
\newtheorem{definition}[cequation]{Definition}
\newtheorem*{definition*}{Definition}

\theoremstyle{remark}
\newtheorem{remark}[cequation]{Remark}

\renewcommand\arraystretch{1.3}

\makeatletter\@addtoreset{equation}{section}
\makeatletter\@addtoreset{section}{part}

\makeatother

\def \CC {\mathbb{C}}

\def \P {\mathbb{P}}
\def \PP {\mathbb{P}}

\def \Z {\mathbb{Z}}

\def \Aff {\mathbb{A}}

\def \lcm {\mathrm{lcm}\,}

\def \ge {\geqslant}
\def \le {\leqslant}

\title{Nef partitions for codimension $2$ weighted complete intersections}

\author{Victor Przyjalkowski and Constantin Shramov}

\address{\emph{Victor Przyjalkowski}
\newline
\textnormal{Steklov Mathematical Institute of RAS, 8 Gubkina street, Moscow 119991, Russia.
}
\newline
\textnormal{National Research University Higher School of Economics, Russian Federation,
Laboratory of Mirror Symmetry, NRU HSE, 6 Usacheva str., Moscow, Russia, 119048.
}
\newline
\textnormal{\texttt{victorprz@mi.ras.ru, victorprz@gmail.com}}}

\address{\emph{Constantin Shramov}
\newline
\textnormal{Steklov Mathematical Institute of RAS,
8 Gubkina street, Moscow 119991, Russia.
}
\newline
\textnormal{National Research University Higher School of Economics, Laboratory of Algebraic Geometry, 6 Usacheva str., Moscow, 119048, Russia.
}
\newline
\textnormal{\texttt{costya.shramov@gmail.com}}}

\thanks{Victor Przyjalkowski was partially supported by Laboratory of Mirror Symmetry NRU HSE, RF Government grant, ag. \textnumero~14.641.31.0001.
Constantin Shramov was supported by the Program of the Presidium of
 the Russian Academy of Sciences \textnumero~01 ``Fundamental Mathematics and
 its Applications'' under grant PRAS-18-01,
by the Russian Academic Excellence Project ``5-100'',
by RFBR grants 15-01-02164 and 15-01-02158, and by the Foundation for the
Advancement of Theoretical Physics and Mathematics ``BASIS''.
Both authors are Young Russian Mathematics award winners and would like to thank
its sponsors and jury.
}

\begin{document}

\begin{abstract}
We prove that a smooth well formed Fano weighted complete intersection of codimension $2$ has a nef partition.
We discuss applications of this fact to Mirror Symmetry.
In particular we list all nef partitions
for smooth well formed Fano weighted complete intersections of dimensions $4$ and $5$ and present weak Landau--Ginzburg
models for them.

\end{abstract}

\maketitle

\section{Introduction}

In~\cite{Gi97} (see also~\cite{HV00}) Givental
defined a Landau--Ginzburg model for a Fano complete intersection $X$
in a smooth toric variety. This Landau--Ginzburg model is a precisely described quasi-projective
family over $\Aff^1$.
Givental proved that an \emph{$I$-series} for~$X$,
that is a generating series of genus $0$ one-pointed Gromov--Witten invariants that count rational curves lying on~$X$,
provides a solution of Picard--Fuchs equation of the Landau--Ginzburg model.
Givental's construction may be used for smooth well formed complete intersections in weighted projective spaces
(as well as it is expected to work for complete intersections in varieties that admit ``good'' toric degenerations
like Grassmannians, see~\cite{Ba04} and~\cite{BCFKS98})
in the same way as for complete intersections in smooth toric varieties,
see~\S\ref{section:NEF} below for details.

The key ingredient in Givental's construction is a notion of \emph{nef partition}.
Let us describe it for the case we are mostly interested in, that is for complete intersections in weighted projective spaces
(we refer the reader to~\cite{Do82} and~\cite{IF00} for the definitions and basic information about weighted projective spaces and complete intersections therein). Let $X$ be a smooth well formed Fano complete intersection of hypersurfaces of degrees $d_1,\ldots,d_c$ in $\PP(a_0,\ldots,a_n)$.

\begin{definition}
\label{definition:nef partition}
A \emph{nef partition} for the complete intersection $X$ is a splitting
$$
\{0,\ldots,n\}=S_0\sqcup S_1\sqcup\ldots\sqcup S_c
$$
such that $\sum_{j\in S_i} a_j=d_i$ for every $i=1,\ldots,c$.
The nef partition is called \emph{nice} if there exists an index $r\in S_0$ such that $a_r=1$.
\end{definition}

Given a nef partition, one can easily write down Givental's Landau--Ginzburg model. Moreover, if the nef partition is nice, one
can birationally represent it as a complex torus with a function on it, that is just a Laurent polynomial, see \S\ref{section:NEF}, which
we call $f_X$.
Such Laurent polynomials are called \emph{weak Landau--Ginzburg models}.
This way of presenting Landau--Ginzburg models has many advantages.
``Good'' weak Landau--Ginzburg models are expected to have
\emph{Calabi--Yau compactifications}. As a result one gets  Landau--Ginzburg models from Homological Mirror Symmetry point of view.
Another expectation is that $f_X$
can be related to a toric degeneration of $X$ via its Newton polytope. If both expectations hold, $f_X$ is called a \emph{toric Landau--Ginzburg model}, see for instance~\cite{Prz13} for
more details.

It appears (see~\S\ref{section:NEF}) that a crucial ingredient for the construction of Givental's Landau--Ginzburg model for a weighted Fano complete intersection
is the existence of a nef partition, and a crucial ingredient for the construction of toric Landau--Ginzburg model
is the existence of a nice nef partition. In~\cite{Prz11} this was shown for complete intersections of Cartier divisors in weighted projective spaces.
In particular,~\cite{Prz11} implies the following.

\begin{theorem}
\label{theorem:hypersurfaces}
Let $X$ be a smooth well formed Fano weighted hypersurface.
Then there exists a nice nef partition for $X$, and $X$ has a toric Landau--Ginzburg model.
\end{theorem}

The main result of the present paper is the following.

\begin{theorem}
\label{theorem:main}
Let $X$ be a smooth well formed Fano weighted complete intersection of codimension $2$.
Then there exists a nice nef partition for $X$.
\end{theorem}

As it is discussed above, this result, together with~\cite{Gi97},~\cite{Prz08},~\cite{Prz13}, and~\cite{ILP13} (see also \S~\ref{section:NEF} below), gives the following.

\begin{corollary}
\label{corollary:main}
In the assumptions of Theorem~\ref{theorem:main} the complete intersection $X$ has a toric Landau--Ginzburg model.
\end{corollary}

Keeping in mind Theorems~\ref{theorem:hypersurfaces} and~\ref{theorem:main} (together with Corollary~\ref{corollary:main}),
we believe that the following is true.

\begin{conjecture}
\label{conjecture:main}
Let $X$ be a smooth well formed Fano weighted complete intersection.
Then there exists a nice nef partition for $X$, and $X$ has a toric Landau--Ginzburg model.
\end{conjecture}

The plan of the paper is as follows.
In~\S\ref{section:NEF} we give definitions of nef partitions and Landau--Ginzburg models they correspond to.
In~\S\ref{section:WPgraphs} we introduce the combinatorial method to
deal with nef partitions based on certain graphs with vertices
labelled by non-trivial weights of the weighted projective space.
In~\S\ref{section:proof} we prove Theorem~\ref{theorem:main} and
make some remarks on its possible generalizations.
In~\S\ref{section:4-folds} we write down nice nef partitions and weak Landau--Ginzburg models
for four- and five-dimensional smooth well formed Fano weighted complete intersections
that are not intersections with linear cones
to give additional evidence for Conjecture~\ref{conjecture:main},
and make a couple of concluding remarks.

\smallskip
We are grateful to the referee for his helpful comments, in particular for his proof of Lemma~\ref{lemma:1-2-3-vertices}(ii)
included in the final version of this paper.

\section{Nef partitions and Landau--Ginzburg models}
\label{section:NEF}

Let $X$ be a smooth well formed Fano complete intersection of hypersurfaces of degrees~\mbox{$d_1,\ldots,d_c$} in $\PP(a_0,\ldots,a_n)$.
Assume that $X$ admits a nef partition
$$
\{0,\ldots,n\}=S_0\sqcup S_1\sqcup \ldots \sqcup S_c.
$$

\begin{definition}
\label{definition:Givental LG}
Givental's Landau--Ginzburg model is a quasi-projective variety in~$(\CC^*)^{n+1}$ with coordinates $x_0,\ldots,x_n$ given by equations
$$\left\{
  \begin{array}{l}
     x_0^{a_0}\cdot\ldots\cdot x_n^{a_n}=1, \\
     \sum_{j\in S_i} x_j=1,\ \ \ i=1,\ldots, c,
  \end{array}
\right.
$$
together with a function $\sum_{j\in S_0} x_j$ called \emph{superpotential}.
\end{definition}

If the nef partition is nice, then one can birationally represent Givental's Landau--Ginzburg model by a complex torus with a function on it.
This function is represented by the following Laurent polynomial.
Let $s_{i,1},\ldots,s_{i,r_i}$, where $r_i=|S_i|$, be elements of $S_i$ and let
$x_{i,1},\ldots,x_{i,r_i}$ be formal variables of weights $a_{s_{i,1}},\ldots,a_{s_{i,r_i}}$.
Since the nef partition is nice, we can assume that $a_{s_{0,r_0}}=1$.
Then Givental's Landau--Ginzburg model for $X$ is birational to $(\CC^*)^{n-c}$ with coordinates $x_{i,j}$
with superpotential
\begin{equation}\label{eq:LG-model}
f_X=\frac{\prod_{i=1}^c (x_{i,1}+\ldots+x_{i,r_i-1}+1)^{d_i}}{\prod\limits_{\substack{i=0,\ldots,c,\\j=1,\ldots,r_i-1}}{x_{i,j}^{a_{i,j}}}}+x_{0,1}+\ldots+x_{0,r_0-1},
\end{equation}
see~\cite{Prz13} and \cite[\S 3]{PSh17}. Indices of variables in the factors in the numerator are~\mbox{$(i,1),\ldots, (i,s_i-1)$}.
However one can choose any $s_i-1$ indices among $(i,1),\ldots, (i,s_i)$ to distinguish such
variables. The resulting family is relatively birational to the one presented above.

\begin{remark}
We see that the main difficulty to represent Givental's Landau--Ginzburg model for a weighted complete  intersection by a Laurent polynomial is to find a nice nef partition; once it is found it is easy to get a birational isomorphism
between Givental's Landau--Ginzburg model and a complex torus, so any nice nef partition gives a Laurent polynomial
in this way. Givental's construction of Landau--Ginzburg models can be applied, besides complete intersections in smooth
toric varieties or weighted projective spaces, to other related cases such as complete intersections in Grassmannians
or partial flag varieties, see~\cite{BCFKS98}. Unlike the case of weighted projective spaces, it is easy
to describe nef partitions in the latter cases, and this can be done in a lot of ways. However the main problem for representing
Landau--Ginzburg models by Laurent polynomials in this case is to find a ``good'' nef partition among all of them,
and to construct the birational isomorphism with a complex torus, see~\cite{PSh17},~\cite{PSh14},~\cite{CKP14},~\cite{PSh15b},~\cite{DH15},~\cite{Pr17}.
\end{remark}

Givental in~\cite{Gi97} computed \emph{$I$-series} of complete intersections in smooth toric varieties,
that is a generating series of genus zero one-pointed Gromov--Witten invariants with descendants.
He proved that this series gives a solution of Picard--Fuchs equation for the family of fibers of the superpotential.
The $I$-series is described in terms of boundary divisors of the toric variety and the hypersurfaces that define
the complete intersection.
In~\cite{Prz07} it was shown that Givental's recipe for $I$-series for complete intersections in singular toric varieties
works in the same way provided that the complete intersection does not intersect the singular locus of the toric variety.
The reason is that curves lying on the complete intersection (that is ones that we count) do not intersect the singular locus,
so we can resolve singularities of the toric variety and apply Givental's recipe; the exceptional divisors do not contribute to
the $I$-series. Thus one can explicitly write down an $I$-series for $X$. One can easily compute the main period
for $f_X$, see, for instance,~\cite{Prz08}, and check that it coincides with the $I$-series for $X$. Moreover,
if the Newton polytope of $f_X$ is reflexive (which
holds for complete intersections in usual projective spaces, see~\cite{Prz16} and~\cite{PSh15a}, but in fact it is not common for weighted complete intersections in weighted projective spaces with non-trivial
weights), then
$f_X$ admits a {Calabi--Yau compactification} (see~\cite[Remark~9]{Prz17}).
The Laurent polynomial $f_X$ also corresponds to a certain toric
degeneration of $X$, see~\cite{ILP13}. In other words, in this case $f_X$ is a \emph{toric Landau--Ginzburg model} of $X$,
see more details, say, in~\cite{Prz13}.

\section{Weighted projective graphs}
\label{section:WPgraphs}

In this section we establish auxiliary combinatorial results that
will be used in the proof of Theorem~\ref{theorem:main}.
Given a graph $\Gamma$, we will denote by $V(\Gamma)$ the set of its vertices.

\begin{definition}\label{definition:WP-graph}
A \emph{weighted projective graph}, or a \emph{WP-graph},
is a non-empty non-oriented graph $\Gamma$ without loops and multiple edges
together with a \emph{weight function}
$$
\alpha_{\Gamma}\colon V(\Gamma)\to \Z_{\ge 2}
$$
such that the following
properties hold
\begin{itemize}
\item for any two vertices $v_1, v_2\in V(\Gamma)$
there exists an edge connecting $v_1$ and $v_2$ in $\Gamma$ if and only if
the numbers
$\alpha_{\Gamma}(v_1)$ and $\alpha_{\Gamma}(v_2)$ are not coprime;

\item for any three vertices $v_1, v_2, v_3\in V(\Gamma)$
the numbers
$\alpha_{\Gamma}(v_1)$, $\alpha_{\Gamma}(v_2)$,
and $\alpha_{\Gamma}(v_3)$ are coprime.
\end{itemize}
\end{definition}

The motivation for Definition~\ref{definition:WP-graph} is as follows. If $\P=\P(a_0,\ldots,a_n)$
is a well formed weighted projective space such that every three
numbers $a_{i_1}$, $a_{i_2}$, and $a_{i_3}$ are coprime, we can produce a
WP-graph whose vertices are labelled by the indices $i$ such that
$a_i>1$, and whose weight function assigns the weight $a_i$ to the corresponding vertex.
We will use this graph to describe singularities of~$\P$ and
complete intersections therein, see~\S\ref{section:proof}.

\begin{definition}
If $\Gamma$ is a WP-graph, we define $\Sigma\Gamma$ to be the sum
of $\alpha_{\Gamma}(v)$ over all vertices~$v$ of~$\Gamma$,
and $\lcm\Gamma$ to be the least common multiple
of $\alpha_{\Gamma}(v)$ over all vertices~$v$ of~$\Gamma$.
\end{definition}

Our current goal is to show that under certain assumptions on a WP-graph
$\Gamma$ one has~\mbox{$\lcm\Gamma\ge \Sigma\Gamma$}. However, this is not always
the case for an arbitrary WP-graph.

\begin{example}
\label{example:exception}
Let $\Delta=\Delta(6,10,15)$ be a graph with three vertices $v_1$, $v_2$, and $v_3$,
and three edges connecting the pairs of the vertices.
Put
$$
\alpha_{\Delta}(v_1)=6,\quad \alpha_{\Delta}(v_2)=10,\quad
\alpha_{\Delta}(v_3)=15,
$$
see Figure~\ref{figure1}.
Then $\Delta$ is a WP-graph with $\Sigma\Delta=31$ and
$\lcm\Delta=30$.
\end{example}

\begin{figure}[htbp]
\begin{center}
\includegraphics[width=3cm]{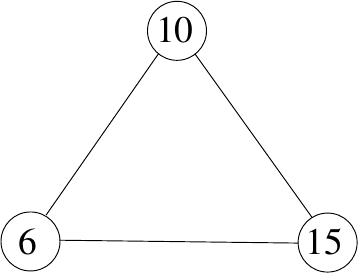}
\end{center}
\caption{WP-graph $\Delta(6,10,15)$}
\label{figure1}
\end{figure}

\begin{remark}
\label{remark:exception}
Suppose that a WP-graph $\Gamma$ contains a WP-subgraph $\Delta(6,10,15)$.
Then it is easy to see that $\Delta(6,10,15)$ is a connected component
of~$\Gamma$, and such subgraph is unique.
\end{remark}

\begin{definition}
Let $\Gamma$ be a WP-graph, and $v$ be its vertex.
We say that $v$ is \emph{weak} if there is an edge
connecting $v$ with another vertex $v'$ of $\Gamma$
such that $\alpha_{\Gamma}(v)$ divides $\alpha_{\Gamma}(v')$.
If $v$ is not weak, we say that it is \emph{strong}.
\end{definition}

\begin{example}
The graph on Figure~\ref{figure2} contains three weak vertices: one labelled
by weight~$7$ and two labelled by weight~$17$.
\end{example}

\begin{figure}[htbp]
\begin{center}
\includegraphics[width=4.5cm]{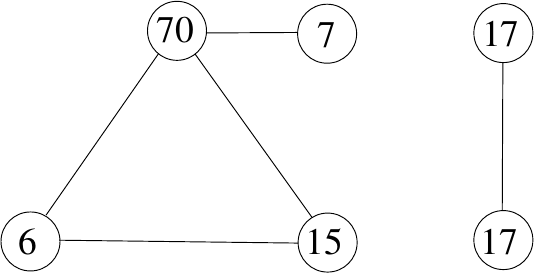}
\end{center}
\caption{Weak and strong vertices}
\label{figure2}
\end{figure}

It easily follows from the definitions that if $v$ is a weak vertex
of a WP-graph $\Gamma$, then there is only one edge in $\Gamma$
containing~$v$. We will see later that (surprisingly) the only WP-graph
$\Gamma$ without weak vertices such that $\lcm\Gamma<\Sigma\Gamma$ is~\mbox{$\Delta(6,10,15)$}.

To proceed we will need the following elementary computation.

\begin{lemma}\label{lemma:elementary}
Let $N$ and $M$ be positive integers such that
$N\ge 4$ and $M\ge\lceil\frac{N-1}{2}\rceil$. Let~\mbox{$a_1,\ldots,a_M$}
be integers such that all $a_i$ are greater than~$1$, and $a_i$ are
be pairwise coprime.
Then~\mbox{$\prod a_i\ge N$}.
\end{lemma}
\begin{proof}
We can assume that $a_M\ge 2$ and
$a_1\ge\ldots\ge a_{M-1}\ge 3$. This implies
that
$$
\prod a_i\ge 2\cdot 3^{M-1}\ge
2\cdot 3^{\lceil\frac{N-3}{2}\rceil}.
$$
The latter value is not smaller than $N$ for
$N\ge 4$, which is easily checked by induction on~$N$.
\end{proof}

\begin{lemma}
\label{lemma:1-2-3-vertices}
Let $\Gamma$ be a connected WP-graph without weak vertices.
The following assertions hold.
\begin{itemize}
\item[(i)] If $\Gamma$ has at most two vertices, then
$\lcm\Gamma\ge \Sigma\Gamma$.

\item[(ii)] If $\Gamma$ has three vertices, then
$\lcm\Gamma\ge \Sigma\Gamma-1$, and
$\lcm\Gamma\ge \Sigma\Gamma$ unless $\Sigma$ is the WP-graph
$\Delta(6,10,15)$.
\end{itemize}
\end{lemma}
\begin{proof}
If $\Gamma$ has only one vertex, then one clearly has
$\lcm\Gamma=\Sigma\Gamma$.

Suppose that $\Gamma$ has two or three vertices, and denote them by $v_i$, $1\le i\le t$, where~$t$ equals either $2$ or $3$.
Put $r_i=\lcm\Gamma/\alpha_{\Gamma}(v_i)$. Then $r_i$ are pairwise coprime positive integers; moreover,
one has $r_i\ge 2$, because $\Gamma$ has no weak vertices.
If $t=2$, then
$$
\Sigma\Gamma=\lcm\Gamma\cdot\left(\frac{1}{r_1}+\frac{1}{r_2}\right)<\lcm\Gamma.
$$
If $t=3$, write
$$
\Sigma\Gamma=\lcm\Gamma\cdot\left(\frac{1}{r_1}+\frac{1}{r_2}+\frac{1}{r_3}\right).
$$
Therefore, one has $\Sigma\Gamma>\lcm\Gamma$ if and only if
$$
\frac{1}{r_1}+\frac{1}{r_2}+\frac{1}{r_3}>1.
$$
This easily implies that $r_1=2$, $r_2=3$, and $r_3=5$ up to permutation,
which in turn means that $\alpha_{\Gamma}(v_1)=15$,
$\alpha_{\Gamma}(v_2)=10$, and $\alpha_{\Gamma}(v_3)=6$.
Therefore, $\Gamma$ is the WP-graph~\mbox{$\Delta(6,10,15)$}, and one has
$$
\lcm\Gamma=30=\Sigma\Gamma-1.
$$
\end{proof}

\begin{lemma}\label{lemma:bez-visyachih-vershin}
Let $\Gamma$ be a connected WP-graph.
Suppose that every vertex of $\Gamma$ is contained in at least two
edges of $\Gamma$. Suppose also that the number $N$ of vertices of
$\Gamma$ is at least~$4$. Then $\lcm\Gamma\ge \Sigma\Gamma$.
\end{lemma}
\begin{proof}
Let $v_{max}$ be the vertex of $\Gamma$ where $\alpha_{\Gamma}$ attains
its maximum. Let $E$ be the set of all edges of $\Gamma$ that do not
contain the vertex~$v_{max}$. It is easy to see
that
$$
|E|\ge \left\lceil\frac{N-1}{2}\right\rceil.
$$

For every edge $e$ connecting the vertices $v_1$ and $v_2$ of $\Gamma$,
let $a_e$ denote the greatest common divisor of
$\alpha_{\Gamma}(v_1)$ and $\alpha_{\Gamma}(v_2)$.
Note that all $a_e$ are pairwise coprime integers,
and all of them are greater than~$1$.
By Lemma~\ref{lemma:elementary}(ii) we have
$$
\lcm\Gamma\ge \alpha_{\Gamma}(v_{max})\cdot\prod\limits_{e\in E} a_e\ge
N\alpha_{\Gamma}(v_{max})\ge \sum\limits_{v\in V(\Gamma)}\alpha_{\Gamma}(v)=
\Sigma\Gamma.
$$
\end{proof}

\begin{lemma}\label{lemma:visyachaya-vershina}
Let $\Gamma$ be a connected WP-graph without weak vertices.
Suppose that there is a vertex $v$ of $\Gamma$ contained in only one
edge of $\Gamma$. Let $\Gamma'$ be the WP-graph that is obtained from
$\Gamma$ by throwing away the vertex $v$ and the edge containing $v$,
and restricting the weight function to the remaining vertices.
Suppose that
$$
\lcm\Gamma'\ge \Sigma\Gamma'-1.
$$
Then
$$
\lcm\Gamma\ge \Sigma\Gamma.
$$
\end{lemma}
\begin{proof}
Let $v'$ be the vertex of $\Gamma$ connected with the vertex $v$.
Write $\alpha_{\Gamma}(v)=ab$ and~\mbox{$\alpha_{\Gamma}(v')=ac$},
where $b$ and $c$ are coprime positive integers,
and $a\ge 2$, see Figure~\ref{figure3}.
\begin{figure}[htbp]
\begin{center}
\includegraphics[width=4.8cm]{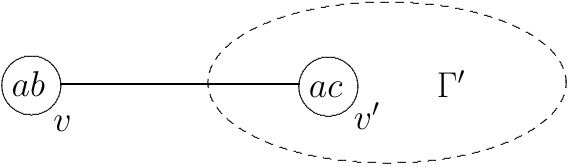}
\end{center}
\caption{Strong vertex contained in a unique edge}
\label{figure3}
\end{figure}

Note that $b\ge 2$ and $c\ge 2$, because $v$ and $v'$ are strong vertices.
One has
$$
\lcm\Gamma=b\lcm\Gamma', \quad
\Sigma\Gamma=ab+\Sigma\Gamma'.
$$
Note also that the graph $\Gamma'$ is connected because the graph $\Gamma$ is connected.

Suppose that $\Sigma\Gamma'\ge ab+2$. Then
$$
\lcm\Gamma=b\lcm\Gamma'\ge 2\lcm\Gamma'\ge
2\Sigma\Gamma'-2\ge \Sigma\Gamma'+ab=\Sigma\Gamma.
$$

Now suppose that $\Sigma\Gamma'\le ab+1$.
This is impossible if $c>b$, because $a\ge 2$.
Thus we have $b\ge c$, which means $b>c$ since $b$ and $c$ are coprime.
Hence
$$
(b-1)c\ge 2b-2\ge b+1.
$$
Note that $\lcm\Gamma'\ge ac$.
This gives
\begin{multline*}
\lcm\Gamma=b\lcm\Gamma'=\lcm\Gamma'+(b-1)\lcm\Gamma'\ge
\Sigma\Gamma'-1+(b-1)ac\ge\\ \ge
\Sigma\Gamma'-1+(b+1)a>
\Sigma\Gamma'+ab=\Sigma\Gamma.
\end{multline*}
\end{proof}

\begin{proposition}
\label{proposition:lcm-connected}
Let $\Gamma$ be a connected WP-graph without weak vertices.
Then
$$
\lcm\Gamma\ge \Sigma\Gamma-1,
$$
and moreover
$\lcm\Gamma\ge \Sigma\Gamma$ unless $\Gamma$
is the WP-graph $\Delta(6,10,15)$.
\end{proposition}
\begin{proof}
We prove the assertion
by induction on the
number $N$ of vertices of $\Gamma$.
We know from Lemma~\ref{lemma:1-2-3-vertices} that
the assertion holds for $N\le 3$.
If $\Gamma$ has a vertex contained in only one
edge of $\Gamma$, then the assertion
follows by induction from Lemma~\ref{lemma:visyachaya-vershina}.
Therefore, we may assume that
$N\ge 4$, and every vertex of $\Gamma$ is contained in at least two
edges of $\Gamma$. Now the assertion follows from
Lemma~\ref{lemma:bez-visyachih-vershin}.
\end{proof}

\begin{corollary}
\label{corollary:lcm}
Let $\Gamma$ be a WP-graph without weak vertices.
Suppose that $\Gamma$
is not the WP-graph $\Delta(6,10,15)$.
Then
$\lcm\Gamma\ge \Sigma\Gamma$.
\end{corollary}
\begin{proof}
Let $\Gamma_1,\ldots,\Gamma_r$ be connected components of
$\Gamma$. Then
\begin{equation}\label{eq:lcm-Sigma-Gamma-i}
\lcm\Gamma=\prod_{i=1}^r\lcm\Gamma_i,\quad \Sigma\Gamma=
\sum\limits_{i=1}^r\Sigma\Gamma_i.
\end{equation}

If a connected component $\Gamma_i$ is not the WP-graph
$\Delta(6,10,15)$, then $\lcm\Gamma_i\ge \Sigma\Gamma_i$ by  Proposition~\ref{proposition:lcm-connected}.
Therefore, if none of $\Gamma_i$ is $\Delta(6,10,15)$, then the assertion immediately follows
from~\eqref{eq:lcm-Sigma-Gamma-i}.

Suppose that some of $\Gamma_i$, say $\Gamma_1$, is
the WP-graph $\Delta(6,10,15)$. Then $r\ge 2$, and none of $\Gamma_2,\ldots,\Gamma_r$
is $\Delta(6,10,15)$. Note that $\Sigma\Gamma_i\ge 2$ (and actually it is at least $7$
for~\mbox{$2\le i\le r$} because of coprimeness condition), so that
$$
30\prod_{i=2}^r\Sigma\Gamma_i\ge 31+\prod_{i=2}^r\Sigma\Gamma_i\ge 31+\sum_{i=2}^r\Sigma\Gamma_i.
$$
Thus~\eqref{eq:lcm-Sigma-Gamma-i} implies
$$
\lcm\Gamma=30\prod_{i=2}^r\lcm\Gamma_i
\ge 30\prod_{i=2}^r\Sigma\Gamma_i\ge
31+\sum_{i=2}^r\Sigma\Gamma_i=\sum_{i=1}^r\Sigma\Gamma_i.
$$
\end{proof}

\begin{definition}\label{definition:WCI-graph}
Let $d_1,\ldots,d_c$ be positive integers.
A \emph{weighted complete intersection graph} (or a \emph{WCI-graph})
\emph{of multidegree $(d_1,\ldots,d_c)$}
is a WP-graph $\Gamma$ such that the following condition holds:
for every $k$ and every choice of
$k$ vertices $v_1,\ldots, v_k$ of $\Gamma$ for which
the greatest common divisor $\delta$ of
$\alpha_{\Gamma}(v_1),\ldots,\alpha_{\Gamma}(v_k)$
is greater than~$1$, there exist $k$ numbers $d_{s_1},\ldots,d_{s_k}$,
$1\le s_1<\ldots<s_k\le c$, whose
greatest common divisor is divisible by $\delta$.
The number $c$ is called the \emph{codimension} of the
WCI-graph~$\Gamma$.
\end{definition}

The motivation for Definition~\ref{definition:WCI-graph} comes from the fact that
a smooth weighted complete intersection
of codimension $1$ or $2$ produces a WCI-graph of codimension $1$ or $2$, respectively,
and some important properties of the weighted complete intersection are controlled
by this WCI-graph, see~\S\ref{section:proof}
for details. Therefore, in this paper we will be mostly interested in WCI-graphs of codimension $1$ and~$2$.

\begin{remark}
It would be more precise to say that a WCI-graph is not just a WP-graph~$\Gamma$ but rather
a collection that consists of~$\Gamma$ and the multidegree $(d_1,\ldots,d_c)$.
In particular, one may have several different WP-graphs with the same~$\Gamma$ and
different multidegrees, and even different codimensions. However, in this paper we are going to deal
only with WCI-graphs of codimension $1$ and $2$, and in any case we want to avoid
this complication of notation and hope that no confusion will arise.
\end{remark}

\begin{lemma}
\label{lemma:two-subgraphs}
Let $\Gamma$ be a WCI-graph of codimension $2$ and bidegree $(d_1,d_2)$.
Then the set of vertices $V(\Gamma)$ is a disjoint union
$$
V(\Gamma)=V_1\sqcup V_2,
$$
such that the complete subgraphs
$\Gamma_1$ and $\Gamma_2$ of $\Gamma$ with vertices $V_1$ and $V_2$
are WP-graphs without weak vertices, none of
$\Gamma_1$ and $\Gamma_2$ contains a connected component $\Delta(6,10,15)$, and
$\lcm\Gamma_i$ divides $d_i$.
\end{lemma}
\begin{proof}
Let $V'\subset V(\Gamma)$ be the set of strong vertices of $\Gamma$,
and $V''=V(\Gamma)\setminus V'$ be the set of weak vertices.
If $\Gamma$ does not
contain a subgraph $\Delta(6,10,15)$, put
$$
V_1'=\{v\in V'\mid \alpha_{\Gamma}(v) \text{\ divides\ } d_1\}.
$$
If $\Gamma$ contains a subgraph $\Delta(6,10,15)$, then
it is easy to see that both $d_1$ and $d_2$ are divisible
by $\lcm\Delta(6,10,15)=30$. In this case we put
$$
V_1'=\{v\in V'\setminus V\big(\Delta(6,10,15)\big)\mid
\alpha_{\Gamma}(v) \text{\ divides\ } d_1\}\cup \{v_1\},
$$
where $v_1$ is an arbitrarily chosen vertex of
$\Delta(6,10,15)$.
We also put $V_2'=V'\setminus V_1'$.
It follows from the definition of a WCI-graph that
for every $v\in V_2'$ the number $\alpha_{\Gamma}(v)$ divides $d_2$.

For every weak vertex $v$ of $\Gamma$ denote by $\tau(v)$ the unique
vertex of $\Gamma$ connected to $v$ by an edge.
It follows from the definition of a WP-graph that either
$\alpha_{\Gamma}(\tau(v))>\alpha_{\Gamma}(v)$, so that
$\tau(v)$ is a strong vertex of $\Gamma$, or
$\alpha_{\Gamma}(\tau(v))=\alpha_{\Gamma}(v)$, so that
$v$ and $\tau(v)$ are both weak vertices. In the latter case the vertices
$v$ and $\tau(v)$ together
with the edge connecting them form a connected component of $\Gamma$
(note however that $v$ and $\tau(v)$ together with the corresponding edge
may form a connected component of $\Gamma$ if $\tau(v)$ is a strong vertex as well).
Let us refer to the former vertices as \emph{weak vertices of the
first type},
and to the latter vertices as \emph{weak vertices of the second type}.
In both cases it follows from the definition of a WCI-graph that
the degrees $d_1$ and $d_2$ are divisible by $\alpha_{\Gamma}(v)$.
Let $V_1''$ be the set of all weak vertices of
the first type such that $\tau(v)\in V_2'$,
and $V_2''$ be the set of all weak vertices of
the first type such that $\tau(v)\in V_1'$.
Finally, let $\tilde{V}_1''$ and $\tilde{V}_2''$ be
sets of weak vertices of the second type
each containing one and
only one vertex from each pair connected by an edge.

Put
$$
V_1=V_1'\cup V_1''\cup \tilde{V}_1'',\quad
V_1=V_2'\cup V_2''\cup \tilde{V}_2''.
$$
Then for
every $v\in V_1$ the number $\alpha_{\Gamma}(v)$
divides $d_1$, and for
every $v\in V_2$ the number $\alpha_{\Gamma}(v)$
divides $d_2$.
The graphs $\Gamma_1$ and $\Gamma_2$ are WP-graphs since they are
complete subgraphs of a WP-graph.
None of them contains a subgraph $\Delta(6,10,15)$; indeed, if one of them does, then
$\Delta(6,10,15)$ is also a subgraph of $\Gamma$, and all three vertices
of $\Delta(6,10,15)$ cannot simultaneously appear as vertices of any of $\Gamma_i$
by construction.
We also see that $\lcm\Gamma_i$ divides~$d_i$.
Moreover, if $v\in V_1$ (respectively,
$v\in V_2$) is a weak vertex of $\Gamma$, then
$\tau(v)\in V_2$ (respectively, $\tau(v)\in V_1$).
This means that the graphs $\Gamma_1$ and $\Gamma_2$ do not
have weak vertices themselves, because any weak vertex of $\Gamma_i$
would also be a weak vertex of~$\Gamma$.
\end{proof}

\begin{example}\label{example:70-17-splitting}
Let $\Gamma$ be a WP-graph from
Figure~\ref{figure2}.
The vertex of $\Gamma$ labelled by weight~$7$ is a weak vertex of the first type,
while the two vertices labelled by weight~$17$ are weak vertices of the second type.
All other vertices are strong.
The WP-graph $\Gamma$ can be considered as a WCI-graph of codimension $2$
and bidegree $(d,d)$, where
$$
d=2\cdot 3\cdot 5\cdot 7\cdot 17=3570.
$$
Following the proof of Lemma~\ref{lemma:two-subgraphs},
one forms the set $V_1'$ that consists of the vertices labelled by $70$, $15$, and $6$,
the set $V_2''$ that consists of the vertex labelled by $7$, the sets
$\tilde{V}_1''$ and $\tilde{V}_2''$ each consisting of one vertex labelled by $17$,
and puts $V_2'=V_1''=\varnothing$.
\end{example}

\begin{corollary}
\label{corollary:codim-2-graphs}
Let $\Gamma$ be a WCI-graph of codimension $2$ and bidegree $(d_1,d_2)$.
Then the set of vertices $V(\Gamma)$ is a disjoint union
$V(\Gamma)=V_1\sqcup V_2$
such that
$$
\sum\limits_{v\in V_i}\alpha_{\Gamma}(v)\le d_i.
$$
\end{corollary}
\begin{proof}
Choose $V_1$ and $V_2$ as in Lemma~\ref{lemma:two-subgraphs},
and let $\Gamma_1$ and $\Gamma_2$ be the complete subgraphs
of $\Gamma$ with vertices $V_1$ and $V_2$. We know that
$d_i$ is divisible by $\lcm\Gamma_i$.
By Corollary~\ref{corollary:lcm}
one has
$$
d_i\ge \lcm\Gamma_i\ge \Sigma\Gamma_i=
\sum\limits_{v\in V_i}\alpha_{\Gamma}(v).
$$
\end{proof}

\begin{example}\label{example:70-17-graphs}
Let $\Gamma$ be a WP-graph from
Figure~\ref{figure2} considered as a WCI-graph of codimension $2$
and bidegree $(3570,3570)$, see Example~\ref{example:70-17-splitting}.
Then one can take $\Gamma_1$ to be the graph with two connected components,
one of them a triangle with vertices labelled by $70$, $15$, and $6$ together with the edges connecting them, and the other a single point labelled by $17$, while $\Gamma_2$ will be a graph with two
connected components, each of them just a single
point, one labelled by $7$ and the other by~$17$.
\end{example}

\section{Proof of the main theorem}
\label{section:proof}

In this section we prove Theorem~\ref{theorem:main} and make some
remarks about its possible generalizations.

\begin{proof}[Proof of Theorem~\ref{theorem:main}]
Let $X$ be a weighted complete intersection of hypersurfaces of degrees $d_1$ and $d_2$ in
$\P(a_0,\ldots,a_n)$.
Since $X$ is smooth and well formed, by \cite[Lemma~2.15]{PrzyalkowskiShramov-Weighted}
for every $k$ and every choice of
$k$ weights
$$
a_{i_1},\ldots, a_{i_k}, 0\le i_1<\ldots<i_k\le n,
$$
whose greatest common divisor $\delta$
is greater than~$1$, there exist $k$ degrees
$$
d_{s_1},\ldots,d_{s_k},
1\le s_1<\ldots<s_k\le 2,
$$
whose greatest common divisor is divisible by $\delta$.
In particular,
any three weights $a_{i_1}, a_{i_2}, a_{i_3}$ are coprime.

We may assume that
$$
1=a_0=\ldots=a_p<a_{p+1}\le\ldots\le a_n.
$$
Let $\Gamma$ be a WP-graph defined as follows.
The vertices of $\Gamma$ are $v_{p+1},\ldots,v_n$,
and two vertices $v_i$ and $v_j$ are connected by an edge
if and only if the weights $a_i$ and $a_j$ are not coprime.
Furthermore, we put $\alpha_{\Gamma}(v_i)=a_i$.
It is easy to see that $\Gamma$ is a WP-graph. Moreover,
$\Gamma$ is a WCI-graph of codimension $2$ and bidegree $(d_1,d_2)$.
By Corollary~\ref{corollary:codim-2-graphs}
there are two disjoint sets $V_1$ and $V_2$ such that
$$
V_1\sqcup V_2=\{p+1,\ldots, n\}
$$
and $\sum_{j\in V_i}a_j\le d_i$ for $i=1,2$.
Since $X$ is Fano, we have
\begin{equation}\label{eq:index}
\sum_{i=0}^n a_j>d_1+d_2,
\end{equation}
see \cite[Theorem 3.3.4]{Do82} or \cite[6.14]{IF00}.
This implies that
one can add the indices of several unit weights, i.e. some indices
from $\{0,\ldots,p\}$, to the sets $V_1$ and $V_2$ to form two disjoint
subsets
$S_1\supset V_1$ and $S_2\supset V_2$
of $\{0,\ldots,n\}$ such that
$\sum_{j\in S_i}a_j=d_i$ for $i=1,2$.
Moreover, since the inequality in~\eqref{eq:index} is strict,
we conclude that the set
$$
S_0=\{0,\ldots,n\}\setminus (S_1\cup S_2)
$$
is not empty. All weights $a_i$ with indices $i\in S_0$ equal $1$, so that the nef partition
$$
\{0,\ldots,n\}=S_0\sqcup S_1\sqcup S_2
$$
is nice.
\end{proof}

\begin{example}\label{example:70-17-partitions}
Let $X$ be a complete intersection of two hypersurfaces of degree~$3570$ in~\mbox{$\P(1^k,6,15,70,7,17,17)$},
where $1^k$ stands for $1$ repeated $k$ times. This is a well formed
Fano weighted complete intersections if $k$ is large enough (and $X$ is general).
Example~\ref{example:70-17-graphs} provides a nice nef partition for~$X$.
Of course, there are many more nice nef partitions in this case. Note that~$X$ is smooth
if it is general enough.
\end{example}

If $X\subset\P(a_0,\ldots,a_n)$ is a smooth well formed Fano weighted hypersurface,
then the corresponding WP-graph $\Gamma$ has no edges at all.
Thus the inequality $\lcm\Gamma\ge\Sigma\Gamma$ is obvious in this case,
and similarly to the proof of Theorem~\ref{theorem:main} we immediately obtain a nice
nef partition for $X$. This recovers the result of Theorem~\ref{theorem:hypersurfaces}.
Also, the proof of Theorem~\ref{theorem:main} gives the following by-product
(cf. \cite[Lemma~3.3]{PrzyalkowskiShramov-Weighted}).

\begin{corollary}
\label{corollary:units}
Let $X$ be a smooth well formed Fano weighted complete intersection of hypersurfaces
$d_1,\ldots,d_c$ in the weighted projective space $\P(a_0,\ldots,a_n)$.
Suppose that~\mbox{$c\le 2$}. Then the number of indices $i\in\{0,\ldots,n\}$
such that $a_i=1$ is at least~\mbox{$I(X)=\sum a_i-\sum d_j$}.
\end{corollary}

Actually, the assertion of Corollary~\ref{corollary:units} holds
in the case of arbitrary codimension, see~\mbox{\cite[Corollary~5.11]{PST}}.

If $X$ is a smooth well formed Calabi--Yau weighted complete intersection
of codimension~$1$ or~$2$, we can argue in the same way as in the proof of Theorem~\ref{theorem:main}
to show that there exists a nef partition for $X$, for which we necessarily have $S_0=\varnothing$ in
the notation of Definition~\ref{definition:nef partition}.
Constructing the dual nef partition we obtain a Calabi--Yau variety $Y$ that is mirror dual to $X$, see~\cite{BB96}. In the same paper
it is proved that the Hodge-theoretic mirror symmetry
holds for $X$ and $Y$. That is, for a given variety $V$ one can define \emph{string Hodge numbers} $h_{st}^{p,q}(V)$
as Hodge numbers of a crepant resolution of $V$ if such resolution exists.
Then, for $n=\dim X=\dim Y$, one has $h_{st}^{p,q}(X)=h_{st}^{n-p,q}(Y)$ provided that the ambient toric variety (weighted projective
space in our case) is Gorenstein.

Finally, we would like to point out a possible approach to a proof
of Conjecture~\ref{conjecture:main} along the lines of the current paper.
If $X$ is a smooth well formed Fano weighted complete intersection
of codimension $3$ or higher in a weighted projective space $\P=\P(a_0,\ldots,a_n)$,
it is possible that some three weights $a_{i_1}$, $a_{i_2}$, and $a_{i_3}$
are not coprime. Thus a WP-graph constructed in the proof of
Theorem~\ref{theorem:main} does not provide an adequate
description of singularities of the weighted projective space~$\P$.
An obvious way to (try to) cope with this
is to replace a graph by a simplicial complex that would remember the greatest common divisors
of arbitrary subsets of weights~$a_i$ in Definition~\ref{definition:WP-graph}. However, this leads to combinatorial difficulties
that we cannot overcome at the moment. Except for the most straightforward ones,
like the effects on weak vertices (which would be not that easy to control) and
possibly larger number of exceptions analogous to our WP-graph~\mbox{$\Delta(6,10,15)$},
there is also a less obvious one (which is in fact easy to deal with). Namely, we
need a finer information about weights and degrees than that provided by
\cite[Lemma~2.15]{PrzyalkowskiShramov-Weighted}.

\begin{example}
Let $X$ be a weighted complete intersection of hypersurfaces of degrees
$2$, $3$, $5$, and $30$ in $\P(1^k,6,10,15)$, where $1^k$ stands for $1$ repeated $k$ times.
Then $X$ is a well formed
Fano weighted complete intersection
provided that $k$ is large and~$X$ is general. Note that the conclusion of
\cite[Lemma~2.15]{PrzyalkowskiShramov-Weighted} holds for $X$. However, it is easy to see
that~$X$ is not smooth. Moreover, there is no nef partition for~$X$.
\end{example}

In any case, it is easy to see that the actual information
one can deduce from the fact that a weighted complete intersection is
smooth is much stronger than that provided by \cite[Lemma~2.15]{PrzyalkowskiShramov-Weighted}.
We also expect that combinatorial difficulties that one has to face on
the way to the proof of Conjecture~\ref{conjecture:main} proposed above are possible
to overcome.

\section{Fano four- and fivefolds}
\label{section:4-folds}

Smooth well formed Fano weighted complete intersections of dimensions $2$ and $3$ are known and well studied (see, for instance,~\cite{IP99}),
as well as their toric Landau--Ginzburg models (see, for instance,~\cite{LP16} and~\cite{Prz13}).
In this section we write down nef partitions and weak Landau--Ginzburg models
for four- and five-dimensional smooth well formed Fano weighted complete intersections. Some of them have codimension greater than~$2$, which gives
additional evidence for Conjecture~\ref{conjecture:main}.
Providing such list is possible due to classification of smooth Fano weighted complete intersections
obtained in~\mbox{\cite[\S5]{PrzyalkowskiShramov-Weighted}}, because finding all nef partitions
for a given complete intersection requires just a simple (though a bit lengthy) computation.

In Tables~\ref{table:Fano-dim-4} and~\ref{table:Fano-dim-5} below we list nef partitions
and corresponding weak Landau--Ginzburg models of four- and five-dimensional smooth well formed Fano
weighted complete intersections that are not intersections with linear cones, see \cite[\S2]{PrzyalkowskiShramov-Weighted} for definitions. These weighted complete
intersections were classified in~\mbox{\cite[\S5]{PrzyalkowskiShramov-Weighted}}, see also \cite[Proposition~2.2.1]{Kuchle-Geography},
where the case of dimension~$4$ was originally established.
In the first column of Tables~\ref{table:Fano-dim-4} and~\ref{table:Fano-dim-5} we put
the number of the family according to tables
in~\mbox{\cite[\S5]{PrzyalkowskiShramov-Weighted}}.
The second column describes the weighted projective spaces
where the weighted complete intersections live.
Here we use the abbreviation
\begin{equation*}
(a_0^{k_0},\ldots,a_m^{k_m})=
(\underbrace{a_0,\ldots,a_0}_{k_0\ \text{times}},\ldots,\underbrace{a_m,\ldots,a_m}_{k_m\ \text{times}}),
\end{equation*}
where $k_0,\ldots,k_m$ are
any positive integers. If some of $k_i$ is equal to $1$ we drop it for simplicity.
In the third column we put the degrees of weighted hypersurfaces that cut out our complete intersections.
The forth column describes nice nef partitions; note that in general there are many of them in every case, but we
do not distinguish between nef partitions obtained by permuting indices corresponding to equal weights. In the fifth column
we write down the corresponding Landau--Ginzburg models.
The latter are obtained using formula~\eqref{eq:LG-model}, where instead of variables
$x_{0,j}$, $x_{1,j}$, $x_{2,j}$, \ldots, we use variables $t_j$, $x_j$, $y_j$, \ldots, respectively, to simplify notation.
We exclude four- and five-dimensional projective spaces (which are complete intersections of codimension~$0$ in themselves) from the tables to unify them with tables from~\mbox{\cite[\S5]{PrzyalkowskiShramov-Weighted}}.

\begin{footnotesize}
\def\arraystretch{1.9}
\begin{center}
\begin{longtable}{||c|c|c|c|c||}
  \hline
  No. & $\P$ & Degrees & Nef partitions & Weak Landau--Ginzburg models
  \\
  \hline
  \hline
  \endhead
\multirow{2}{*}{1} & \multirow{2}{*}{$\PP(1^3,2^2,3^2)$} & \multirow{2}{*}{6,6} &
$\{0\}\sqcup\{1,2,3,4 \}\sqcup\{5,6\}$ &
$\frac{(x_1+x_2+x_3+1)^6(y_1+1)^6}{x_1x_2x_3^2y_1^3}$ \\
& & &
$\{0\}\sqcup\{1,3,5 \}\sqcup\{2,4,6\}$ &
$\frac{(x_1+x_2+1)^6(y_1+y_2+1)^6}{x_1x_2^2y_1y_2^2}$
\\\hline
2 & $\PP(1^4,2,5)$ & 10 & $\{0\}\sqcup\{ 1,2,3,4,5\}$ & $\frac{(x_1+x_2+x_3+x_4+1)^{10}}{x_1x_2x_3x_4^2}$
\\\hline
\multirow{2}{*}{3} & \multirow{2}{*}{$\PP(1^4,2^2,3)$} & \multirow{2}{*}{4,6} &
$\{0\}\sqcup\{ 1,2,4\}\sqcup\{3,5,6\}$ &  $\frac{(x_1+x_2+1)^4(y_1+y_2+1)^6}{x_1x_2y_1y_2^2}$\\
& & & $\{0\}\sqcup\{ 4,5\}\sqcup \{1,2,3,6\}$ & $\frac{(x_1+1)^4(y_1+y_2+y_3+1)^6}{x_1^2y_1y_2y_3}$
\\\hline
4 & $\PP(1^5,4)$ & 8 & $\{0\}\sqcup\{ 1,2,3,4,5\}$ & $\frac{(x_1+x_2+x_3+x_4+1)^8}{x_1x_2x_3x_4}$
\\\hline
5 & $\PP(1^5,2)$ & 6 & $\{0\}\sqcup\{ 1,2,3,4,5\}$ & $\frac{(x_1+x_2+x_3+x_4+1)^6}{x_1x_2x_3x_4}$
\\\hline
\multirow{2}{*}{6} & \multirow{2}{*}{$\PP(1^5,2^2)$} & \multirow{2}{*}{4,4} &
$\{0\}\sqcup\{ 1,2,3,4\}\sqcup\{5,6\}$ & $\frac{(x_1+x_2+x_3+1)^4(y_1+1)^4}{x_1x_2x_3y_1^2}$ \\
& & &
$\{0\}\sqcup\{ 1,2,5\}\sqcup\{3,4,6\}$ & $\frac{(x_1+x_2+1)^4(y_1+y_2+1)^4}{x_1x_2y_1y_2}$
\\\hline
7 & $\PP(1^6,3)$ & 2,6 & $\{0\}\sqcup\{1,2 \}\sqcup \{3,4,5,6\}$ & $\frac{(x_1+1)^2(y_1+y_2+y_3+1)^6}{x_1y_1y_2y_3}$
\\\hline
8 & $\PP^5$ & 5 & $\{0\}\sqcup\{ 1,2,3,4,5\}$ & $\frac{(x_1+x_2+x_3+x_4+1)^4}{x_1x_2x_3x_4}$
\\\hline
\multirow{2}{*}{9} & \multirow{2}{*}{$\PP(1^6,2)$} & \multirow{2}{*}{3,4} &
$\{0\}\sqcup\{1,2,3 \}\sqcup \{4,5,6\}$ & $\frac{(x_1+x_2+1)^3(y_1+y_2+1)^4}{x_1x_2y_1y_1}$\\
& & &
$\{0\}\sqcup\{1,6 \}\sqcup \{2,3,4,5\}$ & $\frac{(x_1+1)^3(y_1+y_2+y_3+1)^4}{x_1y_1y_2y_3}$
\\\hline
10 & $\PP^6$ & 2,4 & $\{0\}\sqcup\{ 1,2\}\sqcup \{3,4,5,6\}$ & $\frac{(x_1+1)^2(y_1+y_2+y_3+1)^4}{x_1y_1y_2y_3}$
\\\hline
11 & $\PP^6$ & 3,3 & $\{0\}\sqcup\{ 1,2,3\}\sqcup \{4,5,6\}$ & $\frac{(x_1+x_2+1)^3(y_1+y_2+1)^3}{x_1x_2y_1y_2}$
\\\hline
12 & $\PP^7$ & 2,2,3 & $\{0\}\sqcup\{1,2\}\sqcup\{3,4\}\sqcup \{5,6,7\}$ & $\frac{(x_1+1)^2(y_1+1)^2(z_1+z_2+1)^3}{x_1y_1z_1z_2}$
\\\hline
13 & $\PP^8$ & 2,2,2,2 & $\{0\}\sqcup\{1,2\}\sqcup \{3,4\}\sqcup \{5,6\}\sqcup \{7,8\}$ & $\frac{(x_1+1)^2(y_1+1)^2(z_1+1)^2(u_1+1)^2}{x_1y_1z_1u_1}$
\\\hline%\hline
14 & $\PP(1^5,3)$ & 6 & $\{0,1\}\sqcup\{ 2,3,4,5\}$ & $\frac{(x_1+x_2+x_3+1)^6}{x_1x_2x_3t_1}+t_1$
\\\hline
15 & $\PP^5$ & 4 & $\{0,1\}\sqcup\{2,3,4,5\}$ & $\frac{(x_1+x_2+x_3+1)^4}{x_1x_2x_3t_1}+t_1$
\\\hline
16 & $\PP^6$ & 2,3 & $\{0,1\}\sqcup\{2,3\}\sqcup\{4,5,6\}$ & $\frac{(x_1+1)^2(y_1+y_2+1)^3}{x_1y_1y_2t_1}+t_1$
\\\hline
17 & $\PP^7$ & 2,2,2 & $\{0,1\}\sqcup\{2,3\}\sqcup\{4,5\}\sqcup\{6,7\}$ & $\frac{(x_1+1)^2(y_1+1)^2(y_1+1)^2}{x_1y_1z_1t_1}+t_1$
\\\hline%\hline
\multirow{2}{*}{18} & \multirow{2}{*}{$\PP(1^4,2,3)$} & \multirow{2}{*}{6} &
$\{0,1,2\}\sqcup\{3,4,5 \}$ & $\frac{(x_1+x_2+1)^6}{x_1x_2^2t_1t_2}+t_1+t_2$\\
& & &
$\{0,4\}\sqcup\{1,2,3,5 \}$ & $\frac{(x_1+x_2+x_3+1)^6}{x_1x_2x_3t_1^2}+t_1$
\\\hline
\multirow{2}{*}{19} & \multirow{2}{*}{$\PP(1^5,2)$} & \multirow{2}{*}{4} &
$\{0,1,2\}\sqcup\{3,4,5 \}$ & $\frac{(x_1+x_2+1)^4}{x_1x_2t_1t_2}+t_1+t_2$\\
& & &
$\{0,5\}\sqcup\{1,2,3,4 \}$ & $\frac{(x_1+x_2+x_3+1)^4}{x_1x_2x_3t_1^2}+t_1$
\\\hline
20 & $\PP^5$ & 3 & $\{0,1,2\}\sqcup\{3,4,5\}$ & $\frac{(x_1+x_2+1)^3}{x_1x_2t_1t_2}+t_1+t_2$
\\\hline
21 & $\PP^6$ & 2,2 & $\{0,1,2\}\sqcup\{3,4\}\sqcup \{5,6\}$ & $\frac{(x_1+1)^2(y_1+1)^2}{x_1y_1t_1t_2}+t_1+t_2$
\\\hline%\hline
22 & $\PP^5$ & 2 & $\{0,1,2,3\}\sqcup\{ 4,5\}$ & $\frac{(x_1+1)^2}{x_1t_1t_2t_3}+t_1+t_2+t_3$
\\
  \hline
\caption[]{\label{table:Fano-dim-4} Fourfold Fano weighted complete intersections.}
\end{longtable}
\end{center}
\end{footnotesize}

\bigskip

\begin{scriptsize}
\def\arraystretch{1.9}
%\begin{center}
\begingroup
\setlength{\LTleft}{-20cm plus -1fill}
\setlength{\LTright}{\LTleft}
\begin{longtable}{||c|c|c|c|c||}
  \hline
  No. & $\P$ & Degrees & Nef partitions & Weak Landau--Ginzburg models
  \\
  \hline
  \hline
\endhead
\multirow{2}{*}{1} & \multirow{2}{*}{$\PP(1^5,2,3,3)$} & \multirow{2}{*}{$6,6$} &
$\{0\}\sqcup \{1,2,3,4,5\}\sqcup\{6,7\}$ & $\frac{(x_1+x_2+x_3+x_4+1)^6(y_1+1)^6}{x_1x_2x_3x_4y_1^3}$\\
& & &
$\{0\}\sqcup \{1,2,3,6\}\sqcup\{4,5,7\}$ & $\frac{(x_1+x_2+x_3+1)^6(y_1+y_2+1)^6}{x_1x_2x_3y_1y_2^2}$
 \\
  \hline
2 & $\PP(1^6,5)$ & $10$ & $\{0\}\sqcup \{1,2,3,4,5,6\}$ & $\frac{(x_1+x_2+x_3+x_4+x_5+1)^{10}}{x_1x_2x_3x_4x_5}$
 \\
  \hline
\multirow{3}{*}{3} & \multirow{3}{*}{$\PP(1^6,2,3)$} & \multirow{2}{*}{$4,6$} &
$\{0\}\sqcup \{1,2,3,4\}\sqcup\{5,6,7\}$ & $\frac{(x_1+x_2+x_3+1)^4(y_1+y_2+1)^6}{x_1x_2x_3y_1y_2^2}$\\
& & &
$\{0\}\sqcup \{1,7\}\sqcup\{2,3,4,5,6\}$ & $\frac{(x_1+1)^4(y_1+y_2+y_3+y_4+1)^6}{x_1y_1y_2y_3y_4}$\\
& & &
$\{0\}\sqcup \{1,2,6\}\sqcup\{3,4,5,7\}$ & $\frac{(x_1+x_2+1)^4(y_1+y_2+y_3+1)^6}{x_1x_2y_1y_2y_3}$
 \\
  \hline
4 & $\PP(1^7,4)$ & $2,8$ & $\{0\}\sqcup \{1,2\}\sqcup\{3,4,5,6,7\}$ & $\frac{(x_1+1)^2(y_1+y_2+y_3+y_4+1)^8}{x_1y_1y_2y_3y_4}$
 \\
  \hline
5 & $\PP^6$ & $6$ & $\{0\}\sqcup \{1,2,3,4,5,6\}$ & $\frac{(x_1+x_2+x_3+x_4+x_5+1)^6}{x_1x_2x_3x_4x_5}$
 \\
  \hline
6 & $\PP(1^7,2)$ & $4,4$ & $\{0\}\sqcup \{1,2,3,4\}\sqcup\{5,6,7\}$ & $\frac{(x_1+x_2+x_3+1)^4(y_1+y_2+1)^4}{x_1x_2x_3y_1y_2}$
 \\
   \hline
7 & $\PP(1^8,3)$ & 2,2,6 & $\{0\}\sqcup \{1,2\}\sqcup\{3,4\}\sqcup\{5,6,7,8\}$ & $\frac{(x_1+1)^2(y_1+1)^2(z_1+z_2+z_3+1)^6}{x_1y_1z_1z_2z_3}$
\\
  \hline
8 & $\PP^7$ & $2,5$ & $\{0\}\sqcup \{1,2\}\sqcup\{3,4,5,6,7\}$ & $\frac{(x_1+1)^2(y_1+y_2+y_3+y_4+1)^5}{x_1y_1y_2y_3y_4}$
 \\
  \hline
9 & $\PP^7$ & $3,4$ & $\{0\}\sqcup \{1,2,3\}\sqcup\{4,5,6,7\}$ & $\frac{(x_1+x_2+1)^3(y_1+y_2+y_3+1)^4}{x_1x_2y_1y_2y_3}$
 \\
   \hline
10 & $\PP^8$ & 2,2,4 & $\{0\}\sqcup \{1,2\}\sqcup\{3,4\}\sqcup\{5,6,7,8\}$ & $\frac{(x_1+1)^2(y_1+1)^2(z_1+z_2+z_3+1)^4}{x_1y_1z_1z_2z_3}$
\\
   \hline
11 & $\PP^8$ & 2,3,3 & $\{0\}\sqcup \{1,2\}\sqcup \{3,4,5\}\sqcup \{6,7,8\}$ & $\frac{(x_1+1)^2(y_1+y_2+1)^3(z_1+z_2+1)^3}{x_1y_1y_2z_1z_2}$
\\
   \hline
12 & $\PP^9$ & 2,2,2,3 & $\{0\}\sqcup \{1,2\}\sqcup\{3,4\}\sqcup\{5,6\}\sqcup\{7,8,9\}$ & $\frac{(x_1+1)^2(y_1+1)^2(z_1+1)^2(u_1+u_2+1)^3}{x_1y_1z_1u_1u_2}$
\\
\hline
13 & $\PP^{10}$ & 2,2,2,2,2 & $\{0\}\sqcup \{1,2\}\sqcup\{3,4\}\sqcup\{5,6\}\sqcup\{7,8\}\sqcup\{9,10\}$ & $\frac{(x_1+1)^2(y_1+1)^2(z_1+1)^2(u_1+1)^2(v_1+1)^2}{x_1y_1z_1u_1v_1}$
\\
  \hline
\multirow{2}{*}{14} & \multirow{2}{*}{$\PP(1^4,2,2,3,3)$} & \multirow{2}{*}{$6,6$} &
$\{0,1\}\sqcup \{2,3,4,5\}\sqcup\{6,7\}$ & $\frac{(x_1+x_2+x_3+1)^6(y_1+1)^6}{x_1x_2x_3^2y_1^3t_1}+t_1$\\
& & &
$\{0,1\}\sqcup \{2,4,6\}\sqcup\{3,5,7\}$ & $\frac{(x_1+x_2+1)^6(y_1+y_2+1)^6}{x_1x_2^2y_1y_2^2t_1}+t_1$
 \\
  \hline
15 & $\PP(1^5,2,5)$ & $10$ & $\{0,1\}\sqcup \{2,3,4,5,6\}$ & $\frac{(x_1+x_2+x_3+x_4+1)^{10}}{x_1x_2x_3x_4^2t_1}+t_1$
 \\
  \hline
\multirow{3}{*}{16} & \multirow{3}{*}{$\PP(1^5,2,2,3)$} & \multirow{3}{*}{$4,6$} &
$\{0,1\}\sqcup \{2,3,5\}\sqcup\{4,6,7\}$ & $\frac{(x_1+x_2+1)^4(y_1+y_2+1)^6}{x_1x_2y_1y_2^2t_1}+t_1$\\
& & &
$\{0,1\}\sqcup \{2,7\}\sqcup\{3,4,5,6\}$ & $\frac{(x_1+1)^4(y_1+y_2+y_3+1)^6}{x_1y_1y_2y_3^2t_1}+t_1$\\
& & &
$\{0,1\}\sqcup \{5,6\}\sqcup\{2,3,4,7\}$ & $\frac{(x_1+1)^4(y_1+y_2+y_3+1)^6}{x_1^2y_1y_2y_3t_1}+t_1$
 \\
  \hline
17 & $\PP(1^6,4)$ & $8$ & $\{0,1\}\sqcup \{2,3,4,5,6\}$ & $\frac{(x_1+x_2+x_3+x_4+1)^8}{x_1x_2x_3x_4t_1}+t_1$
 \\
  \hline
18 & $\PP(1^6,2)$ & $6$ & $\{0,1\}\sqcup \{2,3,4,5,6\}$ & $\frac{(x_1+x_2+x_3+x_4+1)^6}{x_1x_2x_3x_4t_1}+t_1$
 \\
  \hline
\multirow{2}{*}{19} & \multirow{2}{*}{$\PP(1^6,2,2)$} & \multirow{2}{*}{$4,4$} &
$\{0,1\}\sqcup \{2,3,4,5\}\sqcup\{6,7\}$ & $\frac{(x_1+x_2+x_3+1)^4(y_1+1)^4}{x_1x_2x_3y_1^2t_1}+t_1$\\
& & &
$\{0,1\}\sqcup \{2,3,6\}\sqcup\{4,5,7\}$ & $\frac{(x_1+x_2+1)^4(y_1+y_2+1)^4}{x_1x_2y_1y_2t_1}+t_1$
 \\
  \hline
20 & $\PP(1^7,3)$ & $2,6$ & $\{0,1\}\sqcup \{2,3\}\sqcup\{4,5,6,7\}$ & $\frac{(x_1+1)^2(y_1+y_2+y_3+1)^6}{x_1y_1y_2y_3t_1}+t_1$
 \\
  \hline
21 & $\PP^6$ & $5$ & $\{0,1\}\sqcup \{2,3,4,5,6\}$ & $\frac{(x_1+x_2+x_3+x_4+1)^5}{x_1x_2x_3x_4t_1}+t_1$
 \\
  \hline
\multirow{2}{*}{22} & \multirow{2}{*}{$\PP(1^7,2)$} & \multirow{2}{*}{$3,4$} &
$\{0,1\}\sqcup \{2,3,4\}\sqcup\{5,6,7\}$ & $\frac{(x_1+x_2+1)^3(y_1+y_2+1)^4}{x_1x_2y_1y_2t_1}+t_1$\\
& & &
$\{0,1\}\sqcup \{2,7\}\sqcup\{3,4,5,6\}$ & $\frac{(x_1+1)^3(y_1+y_2+y_3+1)^4}{x_1y_1y_2y_3t_1}+t_1$
 \\
  \hline
23 & $\PP^7$ & $2,4$ & $\{0,1\}\sqcup \{2,3\}\sqcup\{4,5,6,7\}$ & $\frac{(x_1+1)^2(y_1+y_2+y_3+1)^4}{x_1y_1y_2y_3t_1}+t_1$
  \\
\hline
24 & $\PP^7$ & 3,3 & $\{0,1\}\sqcup \{2,3,4\}\sqcup\{5,6,7\}$ & $\frac{(x_1+x_2+1)^3(y_1+y_2+1)^3}{x_1x_2y_1y_2t_1}+t_1$
 \\
    \hline
25 & $\PP^8$ & 2,2,3 & $\{0,1\}\sqcup \{2,3\}\sqcup\{4,5\}\sqcup\{6,7,8\}$ & $\frac{(x_1+1)^2(y_1+1)^2(z_1+z_2+1)^3}{x_1y_1z_1z_2t_1}+t_1$
\\
    \hline
26 & $\PP^9$ & 2,2,2,2 & $\{0,1\}\sqcup \{2,3\}\sqcup\{4,5\}\sqcup\{6,7\}\sqcup\{8,9\}$ & $\frac{(x_1+1)^2(y_1+1)^2(z_1+1)^2(u_1+1)^2}{x_1y_1z_1u_1t_1}+t_1$
\\
  \hline%\hline
27 & $\PP(1^6,3)$ & $6$ & $\{0,1,2\}\sqcup \{3,4,5,6\}$ & $\frac{(x_1+x_2+x_3+1)^6}{x_1x_2x_3 t_1t_2}+t_1+t_2$
 \\
  \hline
28 & $\PP^6$ & $4$ & $\{0,1,2\}\sqcup \{3,4,5,6\}$ & $\frac{(x_1+x_2+x_3+1)^4}{x_1x_2x_3 t_1t_2}+t_1+t_2$
 \\
  \hline
29 & $\PP^7$ & 2,3 & $\{0,1,2\}\sqcup \{3,4\}\sqcup\{5,6,7\}$ & $\frac{(x_1+1)^2(y_1+y_2+1)^3}{x_1y_1y_2 t_1t_2}+t_1+t_2$
\\
  \hline
30 & $\PP^8$ & 2,2,2 & $\{0,1,2\}\sqcup \{3,4\}\sqcup\{5,6\}\sqcup\{7,8\}$ & $\frac{(x_1+1)^2(y_1+1)^2(z_1+1)^2}{x_1y_1z_1 t_1t_2}+t_1+t_2$
\\
  \hline%\hline
\multirow{3}{*}{31} & \multirow{3}{*}{$\PP(1^5,2,3)$} & \multirow{3}{*}{$6$} &
$\{0,1,2,3\}\sqcup \{4,5,6\}$ & $\frac{(x_1+x_2+1)^6}{x_1x_2^2 t_1t_2t_3}+t_1+t_2+t_3$\\
& & &
$\{0,1,5\}\sqcup \{2,3,4,6\}$ & $\frac{(x_1+x_2+x_3+1)^6}{x_1x_2x_3 t_1t_2^2}+t_1+t_2$\\
& & &
$\{0,6\}\sqcup \{1,2,3,4,5\}$ & $\frac{(x_1+x_2+x_3+x_4+1)^6}{x_1x_2x_3x_4 t_1^3}+t_1$
 \\
  \hline
\multirow{2}{*}{32} & \multirow{2}{*}{$\PP(1^6,2)$} & \multirow{2}{*}{$4$} &
$\{0,1,2,3\}\sqcup \{4,5,6\}$ & $\frac{(x_1+x_2+1)^4}{x_1x_2 t_1t_2t_3}+t_1+t_2+t_3$\\
& & &
$\{0,1,6\}\sqcup \{2,3,4,5\}$ & $\frac{(x_1+x_2+x_3+1)^4}{x_1x_2x_3 t_1t_2^2}+t_1+t_2$
 \\
  \hline
33 & $\PP^6$ & $3$ & $\{0,1,2,3\}\sqcup \{4,5,6\}$ & $\frac{(x_1+x_2+1)^3}{x_1x_2 t_1t_2t_3}+t_1+t_2+t_3$
 \\
  \hline
34 & $\PP^7$ & $2,2$ & $\{0,1,2,3\}\sqcup \{4,5\}\sqcup\{6,7\}$ & $\frac{(x_1+1)^2(y_1+1)^2}{x_1y_1 t_1t_2t_3}+t_1+t_2+t_3$
 \\
  \hline%\hline
35 & $\PP^6$ & $2$ & $\{0,1,2,3,4\}\sqcup \{5,6\}$ & $\frac{(x_1+1)^2}{x_1 t_1t_2t_3t_4}+t_1+t_2+t_3+t_4$
 \\
  \hline
\caption[]{\label{table:Fano-dim-5} Fivefold Fano weighted complete intersections.}
\end{longtable}
\endgroup
%\end{center}
\end{scriptsize}

\begin{remark}
The set $S_0$ in nef partitions obtained as in the proof of Theorem~\ref{theorem:main} consists of indices only of such variables
that have weight $1$. However some smooth well formed complete intersections may admit other nef partitions,
having non-trivial weights in $S_0$,
see for instance No.~$18$ and $19$ in Table~\ref{table:Fano-dim-4}, and
No.~$31$ and $32$ in Table~\ref{table:Fano-dim-5}.
%More precisely, in Table~\ref{table:Fano-dim-4} these are No.~$18$ and $19$ (the second partition), %while
%in Table~\ref{table:Fano-dim-5} these are
%No.~$31$ (the second and the third partitions) and $32$ (the second partition).
\end{remark}

\begin{question}
One sees that varieties No. $1$, $3$, $6$, $9$, $18$, $19$ from Table~\ref{table:Fano-dim-4} and No. $1$, $14$, $19$, $22$, $32$ from Table~\ref{table:Fano-dim-5} have two different nice nef partitions, while varieties No.~$3$, $16$, and~$31$ from Table~\ref{table:Fano-dim-5} have three
different nice nef partitions. Thus they have two or three weak Landau--Ginzburg models given by these nef partitions.
In~\cite{Li16}  and~\cite{Clarke} (see also~\cite{Pr}) it is proved (under mild additional assumptions) that for complete intersections in Gorenstein toric varieties
Landau--Ginzburg models provided by different nef partitions are birational.
Does this hold for complete intersections in weighted projective spaces?
\end{question}

\begin{remark}
Varieties listed in Tables~\ref{table:Fano-dim-4} and~\ref{table:Fano-dim-5}
admit degenerations to toric varieties whose fan polytopes coincide with Newton polytopes of their weak Landau--Ginzburg models,
see~\cite{ILP13}. Most of them are complete intersections in usual projective spaces. Thus one can prove the existence of (log) Calabi--Yau
compactifications for them, see~\cite{Prz13},~\cite{PSh15a}, and~\cite{Prz17}. Moreover, their existence can be proved
for some other varieties: for variety No. $18$ from Table~\ref{table:Fano-dim-5} using a method from~\cite{PSh15a} and for varieties No. $18$, $19$ (for the second nef partition), $22$ (for the first nef partition), $27$, $32$ (for both nef partitions) from Table~\ref{table:Fano-dim-5}
using a method from~\cite{Prz17}. Thus one can prove that these varieties have toric Landau--Ginzburg models (listed in the last column of the tables).
\end{remark}

\begin{question}
In~\cite{KKP14} Landau--Ginzburg Hodge numbers are defined, see~\cite{LP16} for some discussion on this definition.
Using this definition in~\cite{KKP14} the authors formulated Hodge-theoretic Mirror Symmetry conjecture for Fano varieties
by an analogy with the conjecture for smooth Calabi--Yau varieties.
This conjecture was proved for del Pezzo surfaces in~\cite{LP16}.
One of Hodge numbers can be conjecturally interpreted via number of components of reducible fibers, see~\cite{Prz13} and~\cite{PSh15a}.
In~\cite{PSh15a} this conjecture was checked for complete intersections in usual projective spaces.
Does the Hodge-theoretic Mirror Symmetry conjecture hold for varieties listed in Tables~\ref{table:Fano-dim-4} and~\ref{table:Fano-dim-5}? Does one have an
interpretation via the number of irreducible components of reducible fibers in this case?
Does it hold for all Fano complete intersections in weighted projective spaces having nice nef partitions?
\end{question}

\end{document}